\documentclass[12pt]{amsart}
\usepackage{amsfonts}
\usepackage{color}
\usepackage{soul}
\usepackage[a4paper,top=3cm, bottom=3cm, left=3cm, right=3cm]{geometry}
\begin{document}

\vfuzz2pt 
\newcommand{\red}{\color{red}}
\newcommand{\blue}{\color{blue}}
\newcommand{\black}{\color{black}}

 \newtheorem{thm}{Theorem}[section]
 \newtheorem{cor}{Corollary}[section]
  \newtheorem{con}{Conjecture}[section]
    \newtheorem{cla}{Claim}[section]
 \newtheorem{lem}{Lemma}[section]
 \newtheorem{prop}{Proposition}[section]
 \theoremstyle{definition}
 \newtheorem{defn}{Definition}[section]
 \theoremstyle{remark}
 \newtheorem{rem}{Remark}[section]
 \numberwithin{equation}{section}
\newcommand{\CC}{\mathbb{C}}
\newcommand{\KK}{\mathbb{K}}
\newcommand{\ZZ}{\mathbb{Z}}
\newcommand{\RR}{\mathbb{R}}
\def\a{{\alpha}}

\def\b{{\beta}}

\def\d{{\delta}}

\def\g{{\gamma}}

\def\l{{\lambda}}

\def\gg{{\mathfrak g}}
\def\cal{\mathcal }

\title[Cohomologically rigid solvable Lie superalgebras]{Cohomologically rigid solvable Lie superalgebras with model filiform and model nilpotent nilradical}
\author{S. Bouarroudj}
\address{Sofiane Bouarroudj \newline \indent New York University Abu Dhabi, Division of Science and Mathematics, Po Box: 129188, Abu Dhabi (United Arab Emirates)}
\email{sb3922@nyu.edu}

\author{R.M. Navarro}
\address{Rosa Mar{\'\i}a Navarro \newline \indent
Dpto. de Matem{\'a}ticas, Universidad de Extremadura, C{\'a}ceres
 (Spain) }
\email{rnavarro@unex.es}
\thanks{The second author was supported by the grant NYUAD-065.}


\begin{abstract} In this paper, we find a family $SL^{n,m}$, in any arbitrary dimensions, of cohomologically rigid solvable Lie superalgebras with nilradical the model filiform Lie superalgebra $L^{n,m}$. Moreover, we exhibit a family of cohomologically rigid solvable Lie superalgebras with nilradical the model nilpotent Lie superalgebra of generic characteristic sequence. Both cases correspond to solvable Lie superalgebras of maximal dimension for a given nilradical. Contrariwise, we will show that the family of Lie superalgebras $SL^{n,m}$ can be deformed if defined over a field of odd characteristic.

\end{abstract}

\maketitle
%

\section{Introduction}
The theory of deformation of algebraic structures was initiated by Gerstenhaber \cite{Ger} in the case of associative algebras, then was extended by Nijenhuis and Richardson in the case of Lie algebras (see \cite{NR1, NR2}). Vergne \cite{V} used deformations to study the variety of finite-dimensional Lie algebras. A class of nilpotent Lie algebras baptized {\it filiform} Lie algebra  -- a term coined by Vergne -- was introduced. She proved, among other things, that every filiform Lie algebra can be obtained by a deformation of a model one. Other authors studied  filiform Lie algebras in several contexts, see \cite{ABG, Fia, Goze, GJM, Mill} and references therin. Using Hochshild-Serre Theorem, that turns out to be valid over a field of positive characteristic $p$ as well, Leger and Luks \cite{rigid14} studied rigidity of certain class of solvable Lie algebras in the case where $p\not =2$. Those that are semi-direct sums of nilpotents ideals and abelian subalgebras which act on the ideals in a semi-simple way. 

Deformations of the filiform Lie superalgebras has been studied in a series of  papers, see \cite{Bor07, GKN, A}. Using the representation of $\mathfrak{sl}(2,\Bbb C)$, a complete description  of the deformations has been obtained in \cite{A2}, where an algorithm to obtain all generating 2-cocycles was offered.

Amazingly enough, there is a large class of rigid solvable Lie algebras; for example, it is shown in \cite{rigid_algebras} that there exists a solvable Lie algebra who nilradical is given by an arbitrary characteristic sequence $(n_1,\ldots, n_k,1)$ by which this solvable Lie algebra is rigid. 

In this paper, we will be studying the family of solvable Lie superalgebras $SL^{n,m}$ whose nilradical is  the model filiform Lie superalgebra $L^{n,m}$; for the definition, see \S \ref{filiform}. Moreover, we will extend the study in a natural way  to the maximal-dimensional solvable Lie superalgebras with model nilpotent nilradical, that is $$SN(n_1, \cdots,n_k,1 | m_1, \cdots, m_p).$$  Let us note that all of this study can be considered to be \black  a superization of the results of \cite{rigid_algebras}.

It is well known that the finite dimensional Lie algebras over $\Bbb C$ are rigid, see, e.g. \cite{Fuchs}. Over fields of characteristic $p>0$, there is an exception to this. For $p=3$, the Lie algebra $\mathfrak{o}(5)$ is not rigid, see \cite{Ru, D}. Its deformation has been carried out in \cite{BLW, K, KK, Ru}. For $p=2$, the situation is more complicated, see \cite{BLLS1, Ch}.

The first part of the paper deals with the case when $L^{n,m}$ is defined over a field of characteristic zero. We will show that $SL^{n,m}$ are  rigid Lie superalgebras, see Theorem \ref{Thm_rigid}.  After that, we extend our study to those Lie superalgebras  with model nilpotent nilradical $SN(n_1, \cdots,n_k,1 | m_1, \cdots, m_p)$ showing that they are not only rigid Lie superalgebras but also complete Lie superalgebras - with null dimension for the first and the second cohomology group - see Theorem \ref{thm_Mnilpotent}.

The second part of the paper deals with the case when the family $L^{n,m}$ is defined over a field of characteristisc $p>2$. We will show, first, that this Lie superalgebra is {\it restricted} if and only if $m\leq p$ and $n\leq p+1$, see Theorem \ref{rest}. The notion of restrictedness, introduced by Jacobson in \cite{J}, requires, roughly speaking, the existence of an endomorphism that resembles the pth power $x\mapsto x^p$ in an associative algebras, cf. \cite{strade}. Moreover, we will show that the restricted Lie superalgebra $SL^{n,m}$ can be deformed by means of two 2-cocycles that we exhibit. This is of no surprise since the same phenomena occur for simple Lie (super)algebras

All vector spaces considered in this paper (and thus, all the algebras) are
assumed to be finite dimensional vector spaces over a field of characteristic $p$ (to be specified in each section). Moreover, we shall
use the well-known convention that for the description of a (super) Lie bracket in terms of an ordered basis, only non-vanishing brackets are explicitly exhibited.

\section{Preliminary Results}
Throughout this section, the ground field $\Bbb K$ is either $\Bbb C$ or $\Bbb R$. 
\subsection{The Hochschild-Serre factorization theorem} Although in general the computation of cohomology groups is complicated, the Hochschild-Serre factorization theorem simplifies its computation for semidirect sums of algebras \cite{rigid14}, i.e. 
\[
{\mathfrak{r}}={\mathfrak{t}}\overrightarrow{\oplus} {\mathfrak{n}},
\]
where $\overrightarrow{\oplus}$ stands for the semi-direct sum and satisfying the relations

\[ 
[{\mathfrak{t}},{\mathfrak{n}}]\subset {\mathfrak{n}}, \ \ [{\mathfrak{n}},{\mathfrak{n}}]\subset {\mathfrak{n}}, \ \ [{\mathfrak{t}},{\mathfrak{t}}]\subset {\mathfrak{n}}.
\]
Furthermore, if ${\mathfrak{r}}={\mathfrak{t}}\overrightarrow{\oplus} {\mathfrak{n}}$ is a solvable Lie algebra such that ${\mathfrak{t}}$ is Abelian and the operators $\mathrm{ad}_{ T}$, where $ T \in {\mathfrak{t}},$ are diagonal  then the adjoint cohomology $\mathrm{H}^p({\mathfrak{r}}; {\mathfrak{r}})$ satisfies the following isomorphism

\[
\mathrm{H}^p({\mathfrak{r}}; {\mathfrak{r}}) \simeq \sum_{a+b=p} \mathrm{H}^a({\mathfrak{t}}; \KK) \otimes \mathrm{H}^b({\mathfrak{n}} ; {\mathfrak{r}})^{{\mathfrak{t}}},
\]
where 
\[
\mathrm{H}^b({\mathfrak{n}} ; {\mathfrak{r}})^{{\mathfrak{t}}}:=\{ \varphi \in \mathrm{H}^b({\mathfrak{n}} ; {\mathfrak{r}}) \; \ | \;T \cdot \varphi=0, \text{ for every } T \in {{\mathfrak{t}}}\}
\]
is the space of ${{\mathfrak{t}}}$-invariant cocycles of ${{\mathfrak{n}}}$ with values in ${{\mathfrak{r}}}$; the invariance being defined by:

\[
(T \cdot \varphi) (Z_1, \ldots, Z_b )=[T, \varphi (Z_1, \ldots, Z_b)]- \sum_{s=1}^{b}\varphi (Z_1, \ldots, [T,Z_s], \ldots, Z_b).
\]

\noindent We observe that $\mathrm{H}^a({\mathfrak{t}} ; \KK)=  \bigwedge ^a {\mathfrak{t}}$; hence, $\mathrm{H}^p({\mathfrak{r}} ; {\mathfrak{r}})$ vanishes if and only if $\mathrm{H}^b({\mathfrak{n}} ; {\mathfrak{r}})=0$ holds true for every $b$ such that $0 \leq b \leq p.$ 

In \cite{rigid_algebras}, the authors used the above Hochschild-Serre factorization to compute the cohomology of the following Lie  algebras ${\mathfrak{r}}:={\mathfrak{t}}\overrightarrow{\oplus} {\mathfrak{n}}$ which are semi-direct sums of the algebras ${\mathfrak{t}}$ and ${\mathfrak{n}}$ such that:

\begin{itemize}
\item ${\mathfrak{t}}$ is abelian (i.e. $[{\mathfrak{t}},{\mathfrak{t}}]=0$) and ${\mathfrak{t}}=\mathrm{Span}\{T_1,T_2,..., T_{k+1}\}$;

\item ${\mathfrak{n}}$ is the model nilpotent Lie algebra with arbitrary characteristic sequence $(n_1,n_2, \cdots,n_k,1)$. Equivalently,   ${\mathfrak{n}}$ is the Lie algebra admitting a basis 
\[
\{X_1, \cdots, X_{n_1+1}, \cdots, X_{n_1+n_2+1}, \cdots, 
X_{n_1+\cdots n_k+1} \}
\] such that the only non-vanishing brackets are exactly the following:

\begin{equation}\label{Solva1}
\begin{array}{rcl}
[X_1,X_j]&=&X_{j+1}, \ 2\leq j \leq n_1; \\[1mm]
[X_1,X_{n_1+j}]&=&X_{n_1+1+j}, \ 2\leq j \leq n_2; \\[1mm]
&\vdots  & \\ [1mm]
[X_1,X_{n_1+\cdots +n_{k-2}+j}]&=&X_{n_1+\cdots+n_{k-2}+1+j}, \ 2\leq j \leq n_{k-1}; \\[1mm]
[X_1,X_{n_1+\cdots +n_{k-1}+j}]&=&X_{n_1+\cdots+n_{k-1}+1+j}, \ 2\leq j \leq n_{k}; \\[1mm]
\end{array}
\end{equation}

\item The action of ${\mathfrak{t}}$ over ${\mathfrak{n}}$ corresponds to the action of the maximal torus of derivations, and it is defined by:
\begin{equation}\label{Solva2}
\begin{array}{rcl}
[T_1,X_i]&=&iX_{i}, \ 1\leq i \leq n_1+\cdots+n_k+1; \\[1mm]
[T_2,X_i]&=&X_{i}, \ 2\leq i \leq n_1+1; \\[1mm]
[T_a,X_i]&=&X_{i}, \ 2\leq i \leq n_1+\cdots+n_{a-1}+1, \ 3\leq a \leq k+1; \\[1mm]
\end{array}
\end{equation}
\end{itemize}
Thus, ${\mathfrak{r}}$ is the solvable Lie algebra with basis vectors 
\[
\{T_1,..., T_{k+1}, X_1, \cdots, X_{n_1+1}, \cdots, X_{n_1+n_2+1}, \cdots, 
X_{n_1+\cdots n_k+1} \}
\] 
together with the brackets explicitly described in (\ref{Solva1}) and (\ref{Solva2}). In \cite[Proposition $4$]{rigid_algebras},  the authors showed that $\mathrm{H}^1({\mathfrak{r}} ; {\mathfrak{r}})=0$ and $\mathrm{H}^2({\mathfrak{r}} ; {\mathfrak{r}})=0$ using the Hochschild-Serre factorization theorem for semidirect sums of algebras. 

\subsection{A compendium on Lie superalgebras}

Let us recall now some basic concepts pertaining to Lie superalgebras. 

A {\it superspace} is a vector space with a
$\ZZ_2-$grading: $V=V_{\bar 0} \oplus V_{\bar 1}$. Elements of the
space $V_{\bar 0}$ are called even, and elements of the space $V_{\bar 1} $ are called odd;
the indices $\bar 0$ and $\bar 1$ are taken 
modulo $2$. A linear map $\phi:V\rightarrow W$ between two super
vector spaces is called \textit{even} iff $\phi(V_{\bar 0})\subset W_{\bar 0}$
and $\phi(V_{\bar 1})\subset W_{\bar 1}$ and is called \textit{odd} iff
$\phi(V_{\bar 0})\subset W_{\bar 1}$ and $\phi(V_{\bar 1})\subset W_{\bar 0}$. Clearly,
$\mathrm{Hom}(V,W)=\mathrm{Hom}(V,W)_{\bar 0}\oplus \mathrm{Hom}(V,W)_{\bar 1}$
where the first summand comprises all the even and the second
summand all the odd linear maps. Tensor products $V\otimes W$ are
$\mathbb{Z}_2$ graded by means of $(V\otimes W)_{\bar 0}:=(V_{\bar 0}\otimes
W_{\bar 0})\oplus (V_{\bar 1}\otimes W_{\bar 1})$ and $(V\otimes W)_{\bar 1}:=(V_{\bar 0}\otimes
W_{\bar 1})\oplus (V_{\bar 1}\otimes W_{\bar 0})$.

A {\it Lie superalgebra} (see \cite{Fuchs, Leites, Scheunert}) is a superspace $\gg=\gg_{\bar 0}\oplus\gg_{\bar 1}$, with an even bilinear commutation operation
(or ``supercommutation'') $[~ ,~ ]$, which satisfies the conditions:
\begin{enumerate}
 \item[(i)] $[X,Y]=-(-1)^{\alpha \cdot \beta}[Y,X]\qquad \text{ for all } X\in \gg_{\alpha},
  \text{ and } Y \in \gg_{\beta}$.

 \item[(ii)] $(-1)^{\gamma \cdot \alpha}[X,[Y,Z]]+(-1)^{\alpha \cdot \beta}[Y,[Z,X]]+
 (-1)^{\beta \cdot \gamma}[Z,[X,Y]]=0$
 \newline \indent  \qquad for all $X\in \gg_{\alpha}, Y \in
 \gg_{\beta}, Z
 \in \gg_{\gamma}$ with  $\alpha, \beta, \gamma  \in \ZZ_2$. (Graded Jacobi identity).
\end{enumerate}
Thus, $\gg_{\bar 0}$ is an ordinary Lie algebra, and $\gg_{\bar 1}$ is a module over
$\gg_{\bar 0}$; the Lie superalgebra structure also contains the
symmetric pairing $S^2 \gg_{\bar 1} \longrightarrow \gg_{\bar 0}$, which is actually a
$\gg_{\bar 0}-$homomorphism and satisfies the graded Jacobi identity
applied to three elements of the space $\gg_{\bar 1}$.

The {\it descending central sequence} of a Lie superalgebra
$\gg=\gg_{\bar 0} \oplus \gg_{\bar 1}$ is defined by ${\cal C}^0(\gg)=\gg$, $
 {\cal C}^{k+1}(\gg)=[{\cal C}^k(\gg),\gg]$ for all $k\geq 0$. If
${\cal C}^k(\gg)=\{ 0 \}$ for some $k$, the Lie superalgebra is
called {\it nilpotent}. The smallest integer $k$ such as ${\cal
C}^k(\gg)=\{ 0 \}$ is called the {\it nilindex} of $\gg$.

We define two new {\it descending sequences},
${\cal C}^{k}(\gg_{\bar 0})$ and ${\cal C}^{k}(\gg_{\bar 1})$, as follows:
 ${\cal C}^0(\gg_{i})=~\gg_{i}$, ${\cal C}^{k+1}(\gg_{i})
 =[\gg_0, {\cal C}^k(\gg_{i})]$, \quad
 $k\geq 0$, $i \in \{\bar 0, \bar 1\}$.

 If $\gg=\gg_{\bar 0}\oplus \gg_{\bar 1}$ is a nilpotent Lie
superalgebra, then $\gg$ has super-nilindex or {\it s-nilindex}
$(p,q)$, if the following conditions holds:
\begin{center}
   $({\cal C}^{p-1}(\gg_{\bar 0}))\neq 0$
 \qquad $({\cal C}^{q-1}(\gg_{\bar 1}))\neq 0$, \qquad
 ${\cal C}^{p}(\gg_{\bar 0})={\cal C}^{q}(\gg_{\bar 1})=0.$
\end{center}

In the study of nilpotent Lie algebras we have an invariant called the characteristic sequence that can naturally be  extended to Lie superalgebras.

\begin{defn} For an arbitrary element $x\in \gg_0$, the adjoint operator $\mathrm{ad}_x$ is a nilpotent endomorphism of the space $\gg_i$, where $i\in \{\bar 0, \bar 1\}$. We denote by $gz_i(x)$ the descending sequence of dimensions of Jordan blocks of $\mathrm{ad}_x$. Then, we define the invariant of a Lie
superalgebra $\gg$ as follows:
$$
gz(\gg)=\left(\left.\max_{x\in \gg_{\bar 0} \setminus[\gg_{\bar 0},\gg_0]} gz_0(x) \ \right|
\max_{\widetilde{x}\in \gg_{\bar 0}\setminus[\gg_{\bar 0},\gg_{\bar 0}]} gz_1(\widetilde{x})\right),
$$
where $gz_i$ is in lexicographic order.

The couple $gz(\gg)$ is called characteristic sequence of the Lie superalgebra $\gg$.
\end{defn}

Recall that a \textit{module} $A=A_{\bar 0} \oplus A_{\bar 1}$ of the Lie superalgebra
$\mathfrak{g}$ is an even bilinear map $\mathfrak{g}\times A\to A$ satisfying
\[
X(Y \cdot a)-(-1)^{\alpha\beta}Y(X \cdot a)=[X,Y] \cdot a \quad   \text{ for every $X\in\mathfrak{g}_\alpha, Y\in\mathfrak{g}_\beta, a\in A$.}
\]
Lie superalgebra cohomology is defined in the following well-known
way (see e.g. \cite{Fuchs, Leites2, cohomology}): the superspace of $q$-{\it
dimensional cocycles} of the Lie superalgebra $\gg=\gg_{\bar 0} \oplus
\gg_{\bar 1}$ with coefficients in the $\gg$-module $A=A_{\bar 0} \oplus A_{\bar 1}$ is
given by
\[
 C^q(\gg;A)=\displaystyle{\bigoplus_{q_0+q_1=q} {\rm Hom} \left(\wedge^{q_0}
\gg_{\bar 0} \otimes S^{q_1} \gg_{\bar 1}, A \right)}
\]
This space is graded by $C^q(\gg;A)=C^q_{\bar 0}(\gg;A)\oplus
C^q_{\bar 1}(\gg;A)$ with
\[
  C^q_p(\gg;A)=
  \bigoplus_{\scriptsize \begin{array}{c}q_0+q_1=q \\
                        q_1+r\equiv p \ mod \ 2
                          \end{array}}
  {\rm Hom} \left(\wedge^{q_0}\gg_{\bar 0} \otimes S^{q_1} \gg_{\bar 1}, A_r \right)
\]
The differential $d: C^q(\gg;A) \longrightarrow C^{q+1}(\gg;A)$ is
defined by the formula

\begin{eqnarray*}
 \lefteqn{(dc)\big(g_1,\ldots,g_{q_0},h_1,\ldots,h_{q_1}\big)=}\\
   & &
   \displaystyle \sum_{1\leq s < t \leq q_0} (-1)^{s+t-1}
   c\big([g_s,g_t],g_1,\ldots,\hat{g}_s,\dots,\hat{g}_t,\ldots,g_{q_0},h_1,
        \ldots,h_{q_1}\big) \\
   & & \displaystyle + \sum_{s=1}^{q_0}\sum_{t=1}^{q_1} (-1)^{s-1}
   c\big(g_1,\ldots,\hat{g}_s,\ldots,g_{q_0},[g_s,h_t],h_1,\ldots,\hat{h}_t,
        \ldots,h_{q_1}\big) \\
   & & \displaystyle + \sum_{1 \leq s < t \leq q_1}
   c\big([h_s,h_t],g_1,\ldots,g_{q_0},h_1,\ldots,\hat{h}_s,\ldots,\hat{h}_t,
           \ldots,h_{q_1}\big)\\
   & & \displaystyle + \sum_{s=1}^{q_0} (-1)^s g_s
     \big(c(g_1,\ldots,\hat{g}_s,\ldots,g_{q_0},h_1,\ldots,h_{q_1})\big)\\
   & & \displaystyle + (-1)^{q_0 -1}
    \sum_{s=1}^{q_1}h_s
      \big(c(g_1,\ldots,g_{q_0},h_1,\ldots,\hat{h}_s,\ldots,h_{q_1})\big)
\end{eqnarray*}
where $c \in C^q(\gg;A), \ g_1, \dots, g_{q_0} \in \gg_{\bar 0}$ and
$h_1,\dots,h_{q_1} \in \gg_{\bar 1}$. Obviously, $d \circ d=0$, and
$d(C^q_p(\gg;A))\subset C^{q+1}_p(\gg;A)$  for $q=0,1,2,\dots$ and
$p=\bar 0, \bar 1$. Then we have the {\it cohomology groups}
\[
  \mathrm{H}^q_p(\gg;A):=Z^q_p(\gg;A)\left/ B^q_p(\gg;A) \right.
\]
where the elements of $Z^q_{\bar 0}(\gg;A)$ and $Z^q_{\bar 1}(\gg;A)$ are called
{\it even q-cocycles} and {\it odd q-cocycles}, respectively.
Analogously, the elements of $B^q_{\bar 0}(\gg;A)$ and $B^q_{\bar 1}(\gg;A)$
will be {\it even q-coboundaries} and {\it odd q-coboundaries}, 
respectively.
Two elements of $Z^q(\gg;A)$ are said to be {\it cohomologous} if
their residue classes modulo $B^q(\gg;A)$ coincide, i.e., if their
difference lies in $B^q(\gg;A)$.

\section{maximal-dimensional solvable Lie superalgebras with model filiform nilradical}\label{filiform}
Studying solvable Lie superalgebras, on the other hand, represents more difficulties  than studying  solvable Lie algebras, see \cite{Onindecomposable}. For instance, Lie's theorem does not hold true in general and neither its corollaries.  Therefore, for a solvable Lie superalgebra $\mathfrak r$, the first ideal of the descending central sequence ${\mathfrak r}^2:=[{\mathfrak r}, {\mathfrak r}]$ can not be nilpotent, see  \cite{codimension}. Nevertheless,  in \cite{solvableSA} the authors proved that under the condition of ${\mathfrak r}^2$ being nilpotent, any solvable Lie superalgebra over the real or complex field 
can be obtain by means of outer non-nilpotent derivations of the nilradical in the same way as it occurs for Lie algebras. Therefore, for any solvable Lie superalgebra ${\mathfrak{r}}$ with ${\mathfrak{r}}^2$ nilpotent,  we have a decomposition into semidirect sum:  ${\mathfrak{r}}={\mathfrak{t}}\overrightarrow{\oplus} {\mathfrak{n}}$ such that 
$$ 
[{\mathfrak{t}},{\mathfrak{n}}]\subset {\mathfrak{n}}, \ \ [{\mathfrak{n}},{\mathfrak{n}}]\subset {\mathfrak{n}}, \ \ [{\mathfrak{t}},{\mathfrak{t}}]\subset {\mathfrak{n}}.
$$

In our study we consider as nilradical the model filiform Lie superalgebra $L^{n,m}$, that is,  the simplest filiform Lie superalgebra wich is defined by the only  non-zero bracket products
$$L^{n,m}:
       \left\{\begin{array}{ll}
          [X_1,X_i]=X_{i+1},& 2\leq i \leq n-1\\[1mm]
          [X_1,Y_j]=Y_{j+1},& 1\leq j \leq m-1
       \end{array}\right.
       $$
where $\{ X_1, \dots, X_n \}$ is a basis of $(L^{n,m})_{\bar{0}}$  and $\{ Y_1, \dots, Y_m \}$ is a basis of $(L^{n,m})_{\bar{1}}$. Note that $L^{n,m}$ is the most important filiform Lie superalgebra, in complete analogy to Lie algebras, since all the other filiform
Lie superalgebras can be obtained from it by deformations, see \cite{Bor07}. These infinitesimal deformations are given by even 2-cocycles $Z^2_{\bar 0}(L^{n,m},L^{n,m})$.   

On the other hand, ${\mathfrak{t}}=\mathrm{Span}\{ T_1,T_2,T_3\}$ corresponds to the maximal torus of derivations of $L^{n,m}$. Then  ${\mathfrak{t}}$ is Abelian (i.e., $[{\mathfrak{t}},{\mathfrak{t}}]=0$) and the operators $\mathrm{ad}_T (T \in {\mathfrak{t}})$ are diagonal. A straightforward computation leads to the following action of ${\mathfrak{t}}$ on $L^{n,m}$:
$$\begin{array}{ll}
[T_1,X_i]=iX_{i}, & 1\leq i \leq n; \\[1mm]
 [T_1,Y_j]=jY_{j}, &  1\leq j \leq m; \\[1mm]
[T_2,X_i]=X_{i}, & 2\leq i \leq n; \\[1mm]
[T_3,Y_j]=Y_{j}, & 1\leq j \leq m. \\[1mm]
\end{array}$$
Thus, the solvable Lie superalgebra that we are going to consider, and henceforth named  $SL^{n,m}$, is  defined in a basis  $\{ X_1, \dots, X_n, T_1,T_2,T_3,Y_1, \dots, Y_m \}$ by the only  non-zero brackets
$$
SL^{n,m}:
       \left\{\begin{array}{ll}
          [X_1,X_i]=X_{i+1},& 2\leq i \leq n-1;\\[1mm]
          [X_1,Y_j]=Y_{j+1},& 1\leq j \leq m-1;\\[1mm]
          [T_1,X_i]=iX_{i}, & 1\leq i \leq n; \\[1mm]
          [T_1,Y_j]=jY_{j}, &  1\leq j \leq m; \\[1mm]
[T_2,X_i]=X_{i}, & 2\leq i \leq n; \\[1mm]
[T_3,Y_j]=Y_{j}, & 1\leq j \leq m.\\[1mm]
       \end{array}\right.
       $$
where $\{ X_1, \dots, X_n , T_1,T_2,T_3\}$ is a basis of $(SL^{n,m})_{\bar{0}}$  and $\{ Y_1, \dots, Y_m \}$ is a basis of $(SL^{n,m})_{\bar{1}}$. Recently, this superalgebra has been proved to be the unique, up to isomorphism, maximal-dimensional solvable Lie superalgebra with nilradical the model filiform Lie superalgebra, see \cite{SolvableCuarentena}.

\section{Cohomological rigidity of $SL^{n,m}$}

\begin{prop} \label{even}The maximal-dimensional solvable Lie superalgebra with model filiform nilradical, $SL^{n,m}$,  has trivial even second-cohomology, namely, 
$$
\mathrm{H}^2_{\bar{0}}(SL^{n,m} ; SL^{n,m})\simeq 0.
$$

\end{prop}

\begin{proof}

Let us remark that on account of the lack of symmetric bracket products, $SL^{n,m}$ can be regarded as a Lie algebra (a $\ZZ_2$-graded one). Additionally,  $SL^{n,m}$ as Lie algebra is a rigid one (it corresponds to the case $k=2$ described in   \cite[Proposition $4$]{rigid_algebras}).  Therefore, any $2$-cocycle must be a $2$-coboundary, i.e. any $2$-cocycle $\Psi$ can be expressed as 
$$
\Psi(e_i,e_j)=f([e_i,e_j])-[e_i,f(e_j)]-[f(e_i),e_j],
$$
where $f$ a linear mapping. In particular, any $2$-cocycle $\Psi \in Z^2(SL^{n,m},SL^{n,m})$  (in Lie theory) will be any skew-symmetric bilinear map from $SL^{n,m} \wedge SL^{n,m}$ to $SL^{n,m}$ such that  $d\Psi$ is equivalent to $0$, that is:
$$
\Psi([e_i,e_j],e_k)+\Psi([e_k,e_i],e_j)+\Psi([e_j,e_k],e_i)-[e_i,\Psi(e_j,e_k)]$$ $$-[e_k,\Psi(e_i,e_j)]-[e_j,\Psi(e_k,e_i)]=0,
$$
for all $e_i,e_j,e_k \in SL^{n,m}$.

On the other hand, and considering $SL^{n,m}$ as a Lie superalgebra, we have the following decomposition for the even $2$-cocycles \begin{eqnarray*}
 Z^2_{\bar{0}}(SL^{n,m},SL^{n,m}) & = &
   Z^2_{\bar{0}}(SL^{n,m},SL^{n,m}) ~\cap~
     {\rm Hom} (\gg_{\bar{0}}\wedge \gg_{\bar{0}},\gg_{\bar{0}})\\
     & & ~~\oplus~ Z^2_{\bar{0}}(SL^{n,m},SL^{n,m}) ~\cap~
             {\rm Hom}(\gg_{\bar{0}} \otimes \gg_{\bar{1}},\gg_{\bar{1}})\\
     & & ~~\oplus~Z^2_{\bar{0}}(SL^{n,m},SL^{n,m}) ~\cap~
             {\rm Hom}(S^2 \gg_{\bar{1}},\gg_{\bar{0}})\\
     &=:& A ~\oplus~ B ~\oplus~ C.
\end{eqnarray*}
We shall frequently simplify the notation and write $\gg_{\bar{0}}$ and
$\gg_{\bar{1}}$ instead of $(SL^{n,m})_{\bar{0}}$ and $(SL^{n,m})_{\bar{1}}$, respectively. Next we study $A$, $B$ and $C$ separately. Firstly, regarding $A$ we have that any cocycle $\Psi \in A$ will be any skew-symmetric bilinear map from $\gg_{\bar{0}} \wedge \gg_{\bar{0}}$ to $\gg_{\bar{0}}$ such that $d\Psi$ is equivalent to $0$, that is:

\

$(1)$ \quad $\Psi([e_i,e_j],e_k)+\Psi([e_k,e_i],e_j)+\Psi([e_j,e_k],e_i)-[e_i,\Psi(e_j,e_k)]$

 \quad \quad \quad \quad $-[e_k,\Psi(e_i,e_j)]-[e_j,\Psi(e_k,e_i)]=0$ for all $e_i,e_j,e_k \in \gg_{\bar{0}}$, and 
 
 \

$(2)$ \quad $[Y_j,\Psi(e_i,e_j)]=0, \quad \mbox{ for all } e_i,e_j  \in \gg_{\bar{0}}, \ Y_j \in \gg_{\bar{1}}$.

\

The condition $(1)$ is equivalent to the general condition in Lie algebras to be a cocycle of $Z^2(SL^{n,m},SL^{n,m})\cap  {\rm Hom} (\gg_0\wedge \gg_0,\gg_0)$. But the condition $(2)$ only appears in Lie superalgebras, it does not exist for Lie algebras. 
Anyway, we have that $$A\subset Z^2(SL^{n,m},SL^{n,m})$$ 
therefore  the super-Lie cocycles of $A$ are also Lie cocycles and then can be expressed  $\Psi(e_i,e_j)=f([e_i,e_j])-[e_i,f(e_j)]-[f(e_i),e_j]$ as both Lie and super-Lie coboundaries. A detailed comparison between Lie and super-Lie cocycles of the filiform Lie superalgebra $L^{n,m}$ can be consulted in \cite{A}.

\

With respect to $B$,  analogously as before it can be seen that 
$$B\subset Z^2(SL^{n,m},SL^{n,m})$$ 
consequently  the super-Lie cocycles of $B$ are also Lie cocycles and then Lie and super-Lie coboundaries.

\

$C$ on the other hand, represents a huge difference between Lie algebras and superalgebras because  it is composed entirely by symmetric cocycles and there is no such thing in Lie algebras. By
definition, a cocycle $\varphi$ belonging to $C$  will be a
symmetric bilinear map:
$$
   \varphi: S^2  \gg_{\bar{1}} \longrightarrow  \gg_{\bar{0}}
$$
such that $d \varphi=0$. That is, $\varphi$ will satisfy the two
conditions
\begin{eqnarray*}
   ~[Y_i,\varphi(Y_j,Y_k)]+ [Y_k,\varphi(Y_i,Y_j)]+[Y_j,\varphi(Y_k,Y_i)]
       & = & 0 \quad
                 \text{ for all } Y_i,Y_j,Y_k \in  \gg_{\bar{1}} \\
   ~[e_k,\varphi(Y_i,Y_j)]-\varphi(Y_j,[e_k,Y_i])-\varphi(Y_i,[e_k,Y_j])
       & = & 0 \quad \text{for all } e_k \in  \gg_{\bar{0}} \text{ and } Y_i,Y_j \in  \gg_{\bar{1}}.
\end{eqnarray*}

Taking into account the law of $SL^{n,m}$ and replacing $e_k$ by $T_3$ the second condition above  can
be simplified to
\begin{equation*}
 [T_3,\varphi(Y_i,Y_j)]- \varphi
        ([T_3,Y_i],Y_j)-\varphi(Y_i,[T_3,Y_j])=-2\varphi(Y_i,Y_j)=0, \mbox{ with }
        1 \leq i \leq j \leq m.
\end{equation*}
This leads to $\mathrm{dim}(C)=0$ and then it concludes the proof of the statement of the proposition.
\end{proof}

\begin{prop} \label{odd} The maximal-dimensional solvable Lie superalgebra with model filiform nilradical, $SL^{n,m}$,  has dimension of the odd part of the second group of cohomology equal to zero 
$$\mathrm{dim}(H^2_{\bar{1}}(SL^{n,m};SL^{n,m}))=0.$$

\end{prop}

\begin{proof} Following the spirit of the proof of Proposition \ref{even} we decompose \begin{eqnarray*}
 Z^2_{\bar{1}}(SL^{n,m},SL^{n,m}) & = &
   Z^2_{\bar{1}}(SL^{n,m},SL^{n,m}) ~\cap~
     {\rm Hom} (\gg_{\bar{0}}\wedge \gg_{\bar{0}},\gg_{\bar{1}})\\
     & & ~~\oplus~ Z^2_{\bar{1}}(SL^{n,m},SL^{n,m}) ~\cap~
             {\rm Hom}(\gg_{\bar{0}} \otimes \gg_{\bar{1}},\gg_{\bar{0}})\\
     & & ~~\oplus~Z^2_{\bar{1}}(SL^{n,m},SL^{n,m}) ~\cap~
             {\rm Hom}(S^2 \gg_{\bar{1}},\gg_{\bar{1}})\\
     &=:& A ~\oplus~ B ~\oplus~ C.
\end{eqnarray*}

It can be checked that $A, \ B \subset Z^2(SL^{n,m},SL^{n,m})$. Regarding $C$, we have that any cocycle $\varphi$ belonging to $C$  will be a
symmetric bilinear map:
$$
   \varphi: S^2  \gg_{\bar{1}} \longrightarrow  \gg_{\bar{1}}
$$
such that $d \varphi=0$. That is, $\varphi$ will satisfy the two
conditions
\begin{eqnarray*}
   ~[Y_i,\varphi(Y_j,Y_k)]+ [Y_k,\varphi(Y_i,Y_j)]+[Y_j,\varphi(Y_k,Y_i)]
       & = & 0 \quad
                 \text{for all } Y_i,Y_j,Y_k \in  \gg_{\bar{1}}, \\
   ~[e_k,\varphi(Y_i,Y_j)]-\varphi(Y_j,[e_k,Y_i])-\varphi(Y_i,[e_k,Y_j])
       & = & 0 \quad \text{ for all } e_k \in  \gg_{\bar{0}} \text{ and } Y_i,Y_j \in  \gg_{\bar{1}}.
\end{eqnarray*}

On account of the law of $SL^{n,m}$ and replacing $e_k$ by $T_3$ the second condition of  above  can
be simplified to (here  $1 \leq i \leq j \leq m$):
\begin{equation*}
 [T_3,\varphi(Y_i,Y_j)]- \varphi
        ([T_3,Y_i],Y_j)-\varphi(Y_i,[T_3,Y_j])=\varphi(Y_i,Y_j)-2\varphi(Y_i,Y_j)=0.
\end{equation*}
Therefore, $\mathrm{dim}(C)=0$ and we conclude the proof of the statement of the proposition. 

\end{proof}

Thanks to these two propositions we have the following result.

\begin{thm} \label{Thm_rigid} The maximal-dimensional solvable Lie superalgebra with model filiform nilradical $SL^{n,m}$  is a rigid Lie superalgebra, i.e. the dimension  of the second group of cohomology is equal to zero 
$$\mathrm{dim}(H^2(SL^{n,m};SL^{n,m}))=\mathrm{dim}(H^2_{\bar{0}}(SL^{n,m};SL^{n,m}))+\mathrm{dim}(H^2_{\bar{1}}(SL^{n,m};SL^{n,m}))=0.$$

\end{thm}

\begin{cor}  The maximal-dimensional solvable Lie superalgebra with model filiform nilradical $SL^{n,m}$  is a complete Lie superalgebra,  i.e. 

$$H^i(SL^{n,m};SL^{n,m})=\{0\}, \quad i=0,1,2.$$
\end{cor}

\begin{proof} Since the Lie superalgebra $SL^{n,m}$ is centerless we have $H^0(SL^{n,m},SL^{n,m})=\{0\}$, and as all the superderivations (even and odd) on $SL^{n,m}$ are inner we have $H^1(SL^{n,m},SL^{n,m})=\{0\}$. In particular in \cite[Proof of Theorem 7.1]{SolvableCuarentena} it was shown that any even superderivation $d_{\bar{0}}$ can be expressed  $$d_{\bar{0}}=b_3(\mathrm{ad}_{X_1})-\left(\sum_{k=2}^{n-2} a_{k+1}(\mathrm{ad}_{X_k})\right)-d_n(\mathrm{ad}_{X_n})+ a_1(\mathrm{ad}_{T_1}-2\mathrm{ad}_{T_2}-\mathrm{ad}_{T_3})+b_2(\mathrm{ad}_{T_2})+p_1(\mathrm{ad}_{T_3}).$$
and any odd superderivation $d_{\bar{1}}$ as follows
$$d_{\bar{1}}=-\left(\sum_{k=1}^{m-1} a_{k+1}(\mathrm{ad}_{Y_k})\right)-h_m(\mathrm{ad}_{Y_m}).$$
\end{proof}

\section{maximal-dimensional solvable Lie superalgebras with model nilpotent nilradical}

In a natural way all the study carried out along the last section can be applied to the maximal-dimensional solvable Lie superalgebras with model nilpotent nilradical $SN(n_1, \cdots,n_k,1 | m_1, \cdots, m_p).$ It is  defined in a basis (even $|$ odd) 
\[
\{ x_1, \dots, x_{n_1+\cdots n_k+1}, t_1,\dots,t_{k+1}, t'_1,\dots,t'_p, | y_1, \dots, y_{m_1+\cdots+m_p} \}
\] 
by the only  non-zero bracket products:
$SN(n_1, \cdots,n_k,1 | m_1, \cdots, m_p):$
$$\left\{\begin{array}{ll}
[x_1,x_j]=
-[x_j,x_1]=
x_{j+1}, & 2\leq j \leq n_1; \\[1mm]
[x_1,x_{n_1+\cdots+n_j+i}]=
x_{n_1+\cdots+n_j+i+1}, & 1\leq j \leq k-1, \ 2 \leq i \leq n_{j+1};\\[1mm]
[x_1,y_j]=
-[y_j,x_1]=
y_{j+1}, & 1\leq j \leq m_1-1; \\[1mm]
[x_1,y_{m_1+\cdots+m_j+i}]=
y_{m_1+\cdots+m_j+i+1}, & 1\leq j \leq p-1, \ 1 \leq i \leq m_{j+1}-1;\\[1mm]
[t_1,x_i]=
-[x_i,t_1]=
ix_{i}, & 1\leq i \leq  n_1+\cdots+n_k+1; \\[1mm]
[t_1,y_j]=
-[y_j,t_1]=
jy_{j}, &  1\leq j \leq m_1+\cdots+m_p; \\[1mm]
[t_2,x_i]=
-[x_i,t_2]=
x_{i}, & 2 \leq i \leq n_1+1; \\[1mm]
[t_{j+2},x_{n_1+\cdots+n_j+i}]=
x_{n_1+\cdots+n_j+i}, &  1 \leq j \leq k-1, \  2 \leq i \leq n_{j+1}+1  ; \\[1mm]
[t'_1,y_i]=
-[y_i,t'_1]=
y_{i}, & 1\leq i \leq m_1; \\[1mm]
[t'_{j+1},y_{m_1+\cdots+m_j+i}]=
y_{m_1+\cdots+m_j+i}, &  1 \leq j \leq p-1, \  1 \leq i \leq m_{j+1}.
\end{array}\right.$$
%

Recently in \cite{SolvableCuarentena} it has been proven that all the superderivations are inner and, on the other hand, viewed as a Lie algebra it is rigid. This allows us to develop a similar study as the precedent section, we omit the explicit computation because it is rather cumbersome and it contains no new idea. Thus, we have

\begin{thm}  \label{thm_Mnilpotent} The maximal-dimensional solvable Lie superalgebra with model nilpotent nilradical $SN(n_1, \cdots,n_k,1 | m_1, \cdots, m_p)$  is a complete Lie superalgebra.
\end{thm}

\section{The Modular case}

\subsection{Restrictedness of $SL^{n,m}$}

Hereafter, the ground field $\mathbb{K}$ is of characteristic $p>2$. Following \cite{BKLLS, P}, we say that a Lie superalgebra $\mathfrak{g}$ has a~\textit{$p|2p$-structure} if there exists a~mapping 
\[
[p]:\mathfrak{g}_{\bar 0}\rightarrow \mathfrak{g}_{\bar 0}, \quad a\mapsto a^{[p]} \text{ such that}
\]
\begin{itemize}
\item[(SR1)] \text{$\mathrm{ad}_{a^{[p]}}(b)=(\mathrm{ad}_{a})^p(b)$ \;  for all $a \in \mathfrak{g}_{\bar 0}$ and $b\in \mathfrak{g}$.}
\item[(SR2)] $(\alpha a)^{[p]}=\alpha^p a^{[p]} 
$ \; for all $a\in \mathfrak{g}_{\bar 0}$ and $\alpha \in \mathbb{K}$.\\
\item[(SR3)] $(a+b)^{[p]}=a^{[p]}+b^{[p]}+\sum_{1\leq i\leq p-1}s_i(a,b)$, where the coefficients $s_i(a,b)$ can be obtained from
\[
(\mathrm{ad}_{\lambda a+b})^{p-1}(a)=\sum_{1\leq i \leq p-1} is_i(a,b) \lambda^{i-1}.
\]
\end{itemize}
Recall that the bracket of two odd elements of the Lie superalgebra is polarization of squaring $a\mapsto a^2$. We set 
\[
[2p]:\mathfrak{g}_{\bar 1} \rightarrow \mathfrak{g}_{\bar 0}, \quad a\mapsto (a^2)^{[p]}  \text{~~for any $a\in \mathfrak{g}_{\bar 1}$}.
\]
The pair $(\mathfrak{g}, [p|2p])$ is referred to as a~\textit{restricted} Lie superalgebra. 


Following Jacobson, for a $[p|2p]$ to exists on $\mathfrak g$ it is enough to find a basis $(e_j)_{j\in J}$ of $\mathfrak{g}_{\bar 0}$,  elements $f_j\in \mathfrak{g}_{\bar 0}$ such that  ${(\mathrm{ad}_{e_j})^p=\mathrm{ad}_{f_j}}$. The $p|2p$-mapping $[p|2p]:\mathfrak{g}\rightarrow \mathfrak{g}$ is then given by
\[
e_j^{[p]}=f_j \quad \text{ for all $j\in J$}.
\]
\begin{thm} \label{rest} The Lie superalgebra $SL^{n,m}$, for $n>1$, has a $[p|2p]$-structure if and only if $m\leq p$ and $n\leq p+1$.

\end{thm}

\begin{proof}
In the case where $m\leq p$ and $n\leq p+1$, the $[p|2p]$-structure is given by
\[
T^{[p]}_i =T_i\;  \text{ for $i=1,2,3$; and $X^{[p]}_j=0$\; for $j=1,\ldots,n$}.
\]
In fact, with respect to the ordered basis $\{ X_1, \dots, X_n , T_1,T_2,T_3,Y_1,\dots,Y_m\}$ we have for instance 
\[
\mathrm{ad}_{T_1}=\mathrm{diag}(1,2,3,\dots ,n,0,0,0,1,2,3,\dots ,m).
\] It follows that
\[
(\mathrm{ad}_{T_1})^p=\mathrm{diag}(1^p,2^p,3^p,\dots ,n^p,0,0,0,1^p,2^p,3^p,\dots ,m^p).
\] 
Using Fermat's little theorem $a^p \equiv a$ (mod p) for being $p$ prime, we get 
\[
(\mathrm{ad}_{T_1})^p=\mathrm{ad}_{T_1}.
\]
Now, the result follows by applying Jacobson's Theorem.

Suppose now that $m>p$ or $n>p+1$. We will show that it is not possible to define $X_1^{[p]}$. Indeed, let us write 
\[
X_1^{[p]}=\alpha_1T_1+\alpha_2T_2+\alpha_3T_3+\sum_{j=1}^n \beta_j X_j.
\]
Since $[X_{1}^{[p]},T_1]=-\sum_{j=1}^n j \beta_jX_j$ and $\mathrm{ad}_{X_1}^p(T_1)=0$ it follows that $\beta_j=0$ for every $j=1,\ldots,n$, except perhaps when $j \equiv 0 \mod p$. Now, taking the bracket with respect to $X_1$ we get 
\[
[X_{1}^{[p]},X_1]=\alpha_1X_1-\beta_2X_3-\beta_3X_4-\cdots -\beta_{n-1}X_n.
\]
It follows that $\alpha_1=0$ as well as $\beta_{rp}=0$, except perhaps $\beta_n \not=0$ when $n\equiv 0 \mod p$. 
Similarly, taking the bracket with respect to $X_2$ we get
\[
[X_1^{[p]}, X_2] - \mathrm{ad}_{X_1}^p(X_2)=  \left \{ 
\begin{array}{ll}
 \alpha_2 X_2 - X_{ p+2} & \text{ if $n\geq p+2$}\\[2mm]
  \alpha_2 X_2 & \text{ otherwise}
\end{array}
\right.
\]
The first line already shows a contradiction if $n\geq p+2$. Let us then suppose that $n< p+2$, and therefore, $\alpha_2=0$. Similarly, 
\[
[X_1^{[p]}, Y_1] - \mathrm{ad}_{X_1}^p(Y_1)=  \left \{ 
\begin{array}{ll}
 \alpha_3 Y_1 - Y_{ p+1} & \text{ if $m \geq p+1$}\\[2mm]
 \alpha_3  Y_1 & \text{ otherwise}
\end{array}
\right.
\]
The first line already shows a contradiction if $m \geq  p+1$. The proof is now complete.
\end{proof}

\subsection{The Deform of $SL^{n,m}$} The Lie superalgebra $SL^{n,m}$ is graded as follows
\[
\deg(X_r):=r, \quad \deg(Y_s):=n+s, \quad \deg(T_l)=0.
\]
Having a grading on $SL^{n,m}$ will facilitate the computation of cohomology. 

Since we are only interested in {\it restricted} Lie superalgebras, we will assume that $n\leq p+1$ and $m\leq p$, see Theorem \ref{rest}. 
The following computation has been performed by the mathematica package SuperLie.

\begin{cla}\label{claim1} For $p=3,5,7$ and $11$, the cohomology space $H^2(SL^{n,m}; SL^{n,m})$, for $n>1$, is 
\begin{itemize}
\item[(i)] trivial, for $m<p$ and $n<p+1$; 
\item[(ii)] 1-dimensional, for $m=p$ and $n<p+1$; it it is generated by the 2-cocycle
\begin{equation}\label{coc1}
\mathrm{Deg}=-p :\quad c_1(A,B)=\left \{
\begin{array}{ll}
Y_1& \text{if $(A,B)=(X_1,Y_p)$}\\[2mm]
-Y_1 & \text{if $(A,B)=(Y_p,X_1)$}\\[2mm]
0\quad \text{otherwise}
\end{array}
\right.
\end{equation}
\item[(ii)] 1-dimensional, for $m<p$ and $n=p+1$; it it is generated by the 2-cocycle
\begin{equation}\label{coc2}
\mathrm{Deg}=-p:\quad 
c_2(A,B)=\left \{
\begin{array}{ll}
X_2& \text{if $(A,B)=(X_1,X_{p+1})$}\\[2mm]
-X_2 & \text{if $(A,B)=(X_{p+1},X_1)$}\\[2mm]
0\quad \text{otherwise}
\end{array}
\right.
\end{equation}
\item[(iii)] 2-dimensional, for $m=p$ and $n=p+1$; it is generated by the 2-cocycles (\ref{coc1}) and (\ref{coc2}). 
\end{itemize}
\end{cla}

\begin{cor} Over the ground field $\mathbb{K}$  of characteristic $p$, with $p=3,5,7,11$, the Lie superalgebra $SL^{n,m}$ is a rigid Lie superalgebra provided that $m<p$ and $n<p+1$.
\end{cor}
\begin{con}
Conjecturely, Claim (\ref{claim1}) is true for every $p$.
\end{con}

\begin{rem}
Infinitesimal deformations of symmetric simple modular (and close to simple) Lie superalgebras have been studied in \cite{BGL4}. 
\end{rem}

\end{document}